\documentclass[preprint,12pt]{elsarticle}
\usepackage{amssymb}
\usepackage{amsmath}
\usepackage{amsthm}

\usepackage{hyperref}

\newtheorem{theorem}{Theorem}[section]
\newtheorem{definition}{Definition}[section]

\newtheorem{lemma}{Lemma}[section]
\newtheorem{corollary}{Corollary}[section]
\newtheorem{proposition}{Proposition}[section]
\newtheorem{remark}{Remark}[section]

\journal{Journal Geometry and  Physics}

\begin{document}

\begin{frontmatter}



\title{Riemannian flow techniques on totally geodesic null hypersurfaces} %% Article title


\author[a]{Manuel Gutiérrez\corref{ca}} %% Author name
\cortext[ca]{Corresponding author}
\ead{m_gutierrez@uma.es}

\affiliation[a]{organization={Dto. Álgebra, Geometría y Topología. Universidad de Málaga},%Department and Organization
%          addressline={},
            city={Málaga},
%         postcode={29071},
%        state={},
            country={Spain}}

\author[b,c]{Raymond A. Hounnonkpe}

\ead{rhounnonkpe@ymail.com}

\affiliation[b]{organization={Université d’Abomey-Calavi},%Department and Organization
            city={Abomey-Calavi},
            country={Bénin}}

\affiliation[c]{organization={Institut de Mathématiques et de Sciences Physiques (IMSP)},%Department and Organization
            city={Porto-Novo},
            country={Bénin}}
%% Abstract
\begin{abstract}
We study the influence of the existence of totally geodesic null hypersurface on the properties of a Lorentzian manifold. By coupling the rigging technique with the existence of a null foliation we prove the existence of a Riemann flow structure which allows us to use powerful results to show how curvature conditions on the spacetime restricts its causal structure. We also study the existence of periodic null or spacelike geodesic.
\end{abstract}



%% Keywords

\end{frontmatter}



\section{Introduction}
\label{sec1}
Lorentz geometry has a development which is in principle analogous to that of
Riemannian geometry in many aspects. They share analogous metric structures
with corresponding notion of Levi-Civita connection, curvature, parallel
transport, geodesics, etc. However, the differences are dramatic in some
cases, being one of the most relevant the absence of an analogous to the
Hopf-Rinow theorem in Lorentzian geometry, or more recently the existence of compact Lorentzian manifolds without closed geodesics,
\cite{ABZ}. The interest in studying Lorentzian geometry primarily relies in its use in general relativity. From a geometrical point of view we are interested in exploring the properties and differences of analogous structures in both geometries and studying the new geometric structures arising which has not Riemannian analogy. In fact, submanifolds such that the induced metric is degenerated in some points are a category of objects in
Lorentzian geometry without Riemannian counterpart. The most obvious are the congruence of null geodesics and the family of null hypersurfaces. Both have been studied since a long time ago because they represent physical objects of maximum interest. Null
geodesics represent trajectories of free falling light like particles, and null hypersurfaces represents different kind of light fronts, which includes null cones and black hole horizons. Both are essential in modern
general relativity.

An important characteristic of null hypersupfaces is that they have a privileged direction at each point, the null direction at each tangent space. This means that manifolds furnishing a structure with a distinguished direction at each point can be used as a model to study then. Examples are
all kinds of local twisted (warped, or direct product) spaces with a one-dimensional
factor, contact structures, manifolds with a one-dimensional foliation, etc. The first two ones were studied in \cite{GB,AtiGutHouOle2022}.

It is not a surprise that this strategy works quite well allowing us to unveil valuable information on the geometric structure of these objects itself, and more importantly, showing a tuned link between these properties with the geometric properties of the ambient space.

In this paper, we study some consequences of the presence of totally geodesic null hypersurfaces in a Lorentzian manifold, whose existence is not guaranteed in general. For example, Friedmann spaces do not admit then, see \cite{GB}.  We use the rigging technique introduced by one of the authors in \cite{GB} as a general technique to study null hypersurfaces. We will relate curvature properties with causal properties of the Lorentzian manifold, extending previous results in \cite{AGH2018,AGH2019,AGR,AGH2021}, and we will also provide some kind of conditions to
guarantee the existence of periodic geodesics. The key fact is that the presence of a rigging vector field can be used to show that the null foliation in a totally geodesic null hypersurface becomes a Riemannian flow with bundle-like metric just the rigged metric.

\section{Riemannian flow on totally geodesic null hypersurfaces}

In this section we show that the rigged metric is tunned with the null foliation of a totally geodesic null hypersurface in the sense that it is bundle-like, and the transverse geometry is the geometry of the screen distribution with the rigged metric. This observation makes possible to relate the transverse curvature with the curvature of the ambient space.

We recall the following definition (see \cite{Mol, Tond} for details).

\begin{definition}
Let $M$ be a $n$-dimensional manifold furnished with a $p$-dimensional
foliation $F$ and let $Q=TM/TF$ be the quotient fiber bundle. \textbf{A Riemannian metric }$g_{T}$\textbf{ on }$Q$\textbf{ is
said to be holonomy invariant}, if
\[
L_{X}g_{T}=0,
\]
for any vector field $X$ tangent to $F$. We say that $(M,F)$
is a \textbf{Riemannian foliation}, if $Q$ admits a holonomy invariant metric.
It $p=1$, it is called a \textbf{Riemannian flow}.

\textbf{A Riemannian metric }$g$\textbf{ on }$TM$\textbf{ is said
bundle-like}, if the induced metric on $Q$ is holonomy invariant. In
this case, $Q$ can be identified with the orthogonal complement $TF^{\perp}$.
\end{definition}

\begin{remark}
A flow defined by a Killing vector field is a Riemannian flow \cite{Car}.
\end{remark}

\begin{definition}
\cite[Page 72]{Tond} A Riemannian flow $F$ on $M$ is \textbf{isometric} if
there exists a Riemannian metric $h$ on $M$ and a non singular Killing vector
field $X$ with respect to $h$ such that $F$ is given by the orbits of $X$.
\end{definition}

It is a well-known fact that the flow associated to a lightlike tangent vector
field on a totally geodesic null hypersurface is a Riemannian flow (see
\cite{Zeg1, Zeg}). For completeness, we give here a proof in the presence of a rigging vector field. All the concepts related to the rigging technique may be consulted in \cite{GB}.

\begin{lemma}
Let $(M,g)$ be a Lorentzian manifold and $L$ a totally geodesic null hypersurface with rigging $\zeta$. Then the null foliation in $L$ is a Riemannian flow for which the rigged metric
$\tilde{g}$ is bundle-like. If in addition the rigging $\zeta$ is closed then
the flow is isometric.
\end{lemma}

\begin{proof}
Let $F$ be the $1$-dimensional null foliation defined by the rigged vector field $\xi$. We can identify
$Q=TL/TF$ with the screen bundle $\mathcal{S}^{\zeta}$. From \cite[Proposition
3.7]{GB}, for all $X,Y\in$ $\mathcal{S}^{\zeta}$, $(L_{\xi}\tilde
{g})(X,Y)=-2B(X,Y).$ But since $L$ is totally geodesic, $B$ vanishes
identically so that $(L_{\xi}\tilde{g})(X,Y)=0,\forall\;X,Y\in\mathcal{S}%
^{\zeta}$, which prove that $F$ is a Riemannian flow.

If we suppose that $\zeta$ is closed, then $\xi$ is a Killing vector field
(see \cite[Proposition 3.14]{GB}). Hence the flow is isometric.
\end{proof}

\begin{definition}
\cite{Tond}
\textbf{Transverse connection} and \textbf{transverse curvatures}.

Let $(M,g)$ be a Riemannian manifold and $F$ a codimension $q$ Riemannian
foliation with bundle-like metric $g$. Let $\pi:TM\rightarrow Q$ be the
canonical projection where $Q=TM/TF=TF^{\perp}.$ The basic Levi-Civita
connection or transverse connection) on $Q$ is defined by:
\[
\nabla_{X}^{T}\nu=\left\{
\begin{array}
[c]{ccc}%
\pi([X,Z_{\nu}]) & if & X\in\Gamma(TF)\\
\pi(\nabla_{X}Z_{\nu}) & if & X\in\Gamma(TF^{\perp})
\end{array}
\right.
\]
for any section $\nu\in\Gamma(Q)$, where $Z_{\nu}\in\Gamma(TM)$ is a lift of
$\nu$, i.e. $\pi(Z_{\nu})=\nu$. The basic Levi-Civita connection $\nabla^{T}$
is torsion free and metric-compatible.

If $\mu,\nu$ in $\Gamma(Q)$ are unitary and orthogonal, the tranverse
sectional curvature of the plane spanned by $\mu$ and $\nu$ is defined by
\[
K^{T}(\mu,\nu)=g_{Q}(R^{T}(\mu,\nu)\nu,\mu),
\]
where
\[
R^{T}(\mu,\nu)\nu=\nabla_{\mu}^{T}\nabla_{\nu}^{T}\nu-\nabla_{\nu}^{T}%
\nabla_{\mu}^{T}\nu-\nabla_{\lbrack\mu,\nu]}^{T}\nu.
\]
The tranverse Ricci tensor is defined by
\[
Ric^{T}(X,Y)=\sum_{i=1}^{q}g_{Q}(R^{T}(X,e_{i})e_{i},Y)
\]
for all $X,Y\in\Gamma(Q)$, where $\{e_{i}\}$ is a local orthonormal frame of
$Q$. The transverse Ricci operator is given by
\[
\rho^{T}(X)=\sum_{i=1}^{q}R^{T}(X,e_{i})e_{i}.
\]

The transverse scalar curvature is given by
\[
S^{T}=\sum_{i=1}^{q}g_{Q}(\rho^{T}(e_{i}),e_{i}).
\]
The Riemannian foliation $F$ has \textbf{constant transverse curvature} $k$ if
$K^{T}(P_{x})$ is a constant $k$ for any plane $P_{x}\subset Q_{x}$ and any
$x\in M$. If $k>0$ (respectively $k<0$ ; $k=0$), we say that the Riemannian
foliation is transversally elliptic (respectively hyperbolic, euclidian).
\end{definition}

The following identities are well-known.%
\begin{align}
\nabla_{U}V  & =\nabla_{U}^{L}V+B(U,V)N,\label{e1}\\
\nabla_{U}^{L}X  & =\nabla_{U}^{\ast}X+C(U,X)\xi \label{e2}
\end{align}
where $U,V,\nabla_{U}^{L}V\in\mathfrak{X}(L)$, $X,\nabla_{U}^{\ast}%
X\in\mathcal{S}^{\zeta}$. So they define two connections, $\nabla^{L}$ on $L$
and $\nabla^{\ast}$ on $\mathcal{S}^{\zeta}$.

There are another two identities given by

\begin{align}
\nabla_UN & = \tau (U)N-A(U),\\
\nabla_U\xi & = -\tau(U)\xi-A^*(U),
\end{align}
which defines $A$, the shape operator of $L$ which is a $(1,1)$-tensor in $L$, $A^*$ the shape operator of $\mathcal{S}^{\zeta}$ which is a $(1,1)$-tensor in $L$ taking values in $\mathcal{S}^{\zeta}$, and the rotation $1$-form $\tau$ on $L$ which satisfies

\[
\tau(U)=g(\nabla_UN,\xi)=g(\nabla_U\zeta,\xi).
\]

We prove the following.

\begin{proposition}
\label{equal connection} Let $(M,g)$ be a Lorentzian manifold and $L$ a totally geodesic null hypersurface with rigging $\zeta$. Let $F$ be the Riemannian flow defined by the rigged vector field $\xi$. The transverse connection
$\nabla^{T}$ induced on $Q=TL/TF$ coincides with the connection $\nabla^{\ast}$.
\end{proposition}

\begin{proof}
We can identify $Q=TL/TF$ with the screen bundle $\mathcal{S}^{\zeta}$. From
\cite[Proposition 3.7]{GB}, for all $X,Y\in$ $\mathcal{S}^{\zeta}$,
$\widetilde{\nabla}_{X}Y=\nabla_{X}^{\ast}Y-\widetilde{g}(\widetilde{\nabla
}_{X}\xi,Y)\xi.$ It follows that $\nabla_{X}^{T}Y=\pi(\widetilde{\nabla}%
_{X}Y)=\nabla_{X}^{\ast}Y.$

Now for any $X\in\mathcal{S}^{\zeta}$,
\begin{align}
\nabla_{\xi}^{T}X  &  =\pi([\xi,X])\nonumber\\
&  =\pi(\nabla_{\xi}X-\nabla_{X}\xi)\label{1}\\
&  =\nabla_{\xi}^{\ast}X.\nonumber
\end{align}
where we can use $[\xi,X]=\nabla_{\xi}X-\nabla_{X}\xi$ because both
$\nabla_{\xi}X,\nabla_{X}\xi\in\mathfrak{X}(L)$. The last equality comes from
the fact that $\nabla_{\xi}X=\nabla_{\xi}^{\ast}X+C(\xi,X)\xi$ and $\nabla
_{X}\xi=-A^{\ast}(X)-\tau(X)\xi=-\tau(X)\xi$, where we use $B=0$, see \cite{GB}.
\end{proof}

As a consequence, we get the following.

\begin{proposition}
\label{curvat} Let $(M,g)$ be a Lorentzian manifold and
$L$ a totally geodesic null hypersurface with rigging $\zeta$. Let $F$ be the Riemannian flow defined by the rigged vector field $\xi$. Let $\nabla^{T}$ be the transverse
connection induced on $Q=\mathcal{S}^{\zeta}$. For any $X,Y\in\mathcal{S}%
^{\zeta}$ unitary and orthogonal, we have
\[
K^{T}(X,Y)=\widetilde{K}(X,Y)+\frac{3}{4}d\omega(X,Y)^{2}=K(X,Y).
\]

\end{proposition}

\begin{proof}
Let $X,Y\in\mathcal{S}^{\zeta}$ be unitary and orthogonal. From Proposition
\ref{equal connection}, $K^{T}(X,Y)=\widetilde{g}_{Q}(R^{\ast}(X,Y)Y,X)$,
where
\[
R^{\ast}(X,Y)Y=\nabla_{X}^{\ast}\nabla_{Y}^{\ast}Y-\nabla_{Y}^{\ast}\nabla
_{X}^{\ast}Y-\nabla_{\lbrack X,Y]}^{\ast}Y.
\]

From \cite[Proposition 3.7]{GB},
\[
\widetilde{\nabla}_{X}Y=\nabla_{X}^{\ast}Y-\widetilde{g}(\widetilde{\nabla
}_{X}\xi,Y)\xi\ \ and\ \ (L_{\xi}\tilde{g})(X,Y)=-2B(X,Y).
\]

It follows that $\nabla_{X}^{\ast}\nabla_{Y}^{\ast}Y=\widetilde{\nabla}%
_{X}\nabla_{Y}^{\ast}Y+\widetilde{g}(\widetilde{\nabla}_{X}\xi,\nabla
_{Y}^{\ast}Y)\xi$. But $\nabla_{Y}^{\ast}Y=\widetilde{\nabla}_{Y}%
Y+\widetilde{g}(\widetilde{\nabla}_{Y}\xi,Y)\xi.$ Since $(L_{\xi}\tilde
{g})(Y,Y)=-2B(Y,Y)=0$ then $\widetilde{g}(\widetilde{\nabla}_{Y}\xi,Y)=0$. So,
$\nabla_{Y}^{\ast}Y=\widetilde{\nabla}_{Y}Y$ and we get
\begin{equation}
\widetilde{g}_{Q}(\nabla_{X}^{\ast}\nabla_{Y}^{\ast}Y,X)=\widetilde
{g}(\widetilde{\nabla}_{X}\widetilde{\nabla}_{Y}Y,X). \label{equ 1}%
\end{equation}

In the same way, we have $\nabla_{Y}^{\ast}\nabla_{X}^{\ast}Y=\widetilde
{\nabla}_{Y}\nabla_{X}^{\ast}Y+\widetilde{g}(\widetilde{\nabla}_{Y}\xi
,\nabla_{X}^{\ast}Y)\xi$. Using the fact that $\nabla_{X}^{\ast}%
Y=\widetilde{\nabla}_{X}Y+\widetilde{g}(\widetilde{\nabla}_{X}\xi,Y)\xi$, we
get
\[
\nabla_{Y}^{\ast}\nabla_{X}^{\ast}Y=\widetilde{\nabla}_{Y}\widetilde{\nabla
}_{X}Y+Y\left(  \widetilde{g}(\widetilde{\nabla}_{X}\xi,Y)\right)
\xi+\widetilde{g}(\widetilde{\nabla}_{X}\xi,Y)\widetilde{\nabla}_{Y}%
\xi+\widetilde{g}(\widetilde{\nabla}_{Y}\xi,\widetilde{\nabla}_{X}Y)\xi.
\]

It follows that
\begin{equation}
\widetilde{g}_{Q}(\nabla_{Y}^{\ast}\nabla_{X}^{\ast}Y,X)=\widetilde
{g}(\widetilde{\nabla}_{Y}\widetilde{\nabla}_{X}Y,X)+\widetilde{g}%
(\widetilde{\nabla}_{X}\xi,Y)\widetilde{g}(\widetilde{\nabla}_{Y}\xi,X).
\label{equ 2}%
\end{equation}

Now using the decomposition $[X,Y]=[X,Y]^{\zeta}+\widetilde{g}([X,Y],\xi)\xi$
where $[X,Y]^{\zeta}$ is the projection of $[X,Y]$ on $\mathcal{S}^{\zeta}$,
we have
\begin{align}
\nabla_{\lbrack X,Y]}^{\ast}Y & =\nabla_{\lbrack X,Y]^{\zeta}}^{\ast}%
Y+\nabla_{\widetilde{g}([X,Y],\xi)\xi}^{\ast}Y \\ & =\widetilde{\nabla
}_{[X,Y]^{\zeta}}Y+\widetilde{g}(\widetilde{\nabla}_{[X,Y]^{\zeta}}\xi
,Y)\xi+\widetilde{g}([X,Y],\xi)\nabla_{\xi}^{\ast}Y.
\end{align}

Using that $L$ is totally geodesic and (\ref{e2}) it follows that
\[
\nabla_{\lbrack X,Y]}^{\ast}Y=\widetilde{\nabla}_{[X,Y]^{\zeta}}%
Y+\widetilde{g}(\widetilde{\nabla}_{[X,Y]^{\zeta}}\xi,Y)\xi+\widetilde
{g}([X,Y],\xi)[\xi,Y]^{\zeta}.
\]

Hence
\[
\widetilde{g}_{Q}(\nabla_{\lbrack X,Y]}^{\ast}Y,X)=\widetilde{g}%
(\widetilde{\nabla}_{[X,Y]^{\zeta}}Y,X)+\widetilde{g}([X,Y],\xi)\widetilde
{g}([\xi,Y],X).
\]

From this, we get
\begin{align*}
\widetilde{g}_{Q}(\nabla_{\lbrack X,Y]}^{\ast}Y,X) & =\widetilde{g}%
(\widetilde{\nabla}_{[X,Y]^{\zeta}}Y,X)+\widetilde{g}([X,Y],\xi)\widetilde
{g}(\widetilde{\nabla}_{\xi}Y,X) \\ & -\widetilde{g}([X,Y],\xi)\widetilde
{g}(\widetilde{\nabla}_{Y}\xi,X).
\end{align*}

Then%
\begin{equation}
\widetilde{g}_{Q}(\nabla_{\lbrack X,Y]}^{\ast}Y,X)=\widetilde{g}%
(\widetilde{\nabla}_{[X,Y]}Y,X)-\widetilde{g}([X,Y],\xi)\widetilde
{g}(\widetilde{\nabla}_{Y}\xi,X). \label{equ 3}%
\end{equation}

From (\ref{equ 1}), (\ref{equ 2}) and (\ref{equ 3}), we get
\[
K^{T}(X,Y)=\widetilde{K}(X,Y)-\widetilde{g}(\widetilde{\nabla}_{X}%
\xi,Y)\widetilde{g}(\widetilde{\nabla}_{Y}\xi,X)+\widetilde{g}([X,Y],\xi
)\widetilde{g}(\widetilde{\nabla}_{Y}\xi,X).
\]

Since $(L_{\xi}\tilde{g})(X,Y)=-2B(X,Y)=0,$ we have $\widetilde{g}%
(\widetilde{\nabla}_{X}\xi,Y)=-\widetilde{g}(\widetilde{\nabla}_{Y}\xi,X)$.
Let $\omega$ be the rigged 1-form, that is, the 1-form $\widetilde{g}$-equivalent to $\xi$.  We have $d\omega(X,Y)=\widetilde{g}(\widetilde{\nabla}_{X}\xi,Y)-\widetilde
{g}(\widetilde{\nabla}_{Y}\xi,X)$, so
\[
\widetilde{g}(\widetilde{\nabla}_{X}\xi,Y)=-\widetilde{g}(\widetilde{\nabla
}_{Y}\xi,X)=\frac{1}{2}d\omega(X,Y).
\]

Elsewhere,
\[
\widetilde{g}([X,Y],\xi)=\widetilde{g}(\widetilde{\nabla}_{X}Y,\xi
)-\widetilde{g}(\widetilde{\nabla}_{Y}X,\xi)=-\widetilde{g}(\widetilde{\nabla
}_{X}\xi,Y)+\widetilde{g}(\widetilde{\nabla}_{Y}\xi,X).
\]

Hence
\[
\widetilde{g}([X,Y],\xi)=-2\widetilde{g}(\widetilde{\nabla}_{X}\xi
,Y)=-d\omega(X,Y).
\]

Finally we get
\[
K^{T}(X,Y)=\widetilde{K}(X,Y)+\frac{3}{4}d\omega(X,Y)^{2}.
\]

But from \cite[Theorem 4.2]{GB}, $\widetilde{K}(X,Y)+\frac{3}{4}%
d\omega(X,Y)^{2}=K(X,Y).$ In conclusion,
\[
K^{T}(X,Y)=\widetilde{K}(X,Y)+\frac{3}{4}d\omega(X,Y)^{2}=K(X,Y).
\]

\end{proof}

\begin{remark}
\label{rmq closed} 1. If $(M,g)$ has positive (resp.
negative) curvature on spacelike planes, then the flow has positive (resp.
negative) transverse curvature.\newline2. If the rigging $\zeta$ is closed,
then $d\omega=0$ and we get
\[
K^{T}(X,Y)=\widetilde{K}(X,Y)=K(X,Y).
\]

\end{remark}

\begin{corollary}
\label{transv constant} Let $(M,g)$ be a Lorentzian manifold of constant sectional curvature $c$ and $L$ a totally geodesic null hypersurface with rigging $\zeta$. Then the Riemannian flow $F$ defined by the rigged vector field $\xi$ has constant transverse curvature $c$. In particular, the flow is transversally elliptic (resp. transversally hyperbolic, transversally
euclidian) if $c>0$ (respectively $c<0$ , $c=0$).
\end{corollary}

Flows with constant transverse curvature on a compact Riemannian manifold $M$
have been classified, see \cite[Theorem 4.1, Theorem 4.2, Theorem 4.3]{Mol}.

In the hyperbolic case, there are two possibilities:\newline1) All the orbits
of $F$ are closed and constitute a Seifert fibration.\newline2) None of the orbits
of $F$ are closed. Then $dim (M) = 3 $ or $4$ and the flow is conjugated to a
particular flow, \cite[Theorem 4.1]{Mol}.

In the euclidean case, there are two possibilities:\newline1) All the orbits
of $F$ are closed and constitute a Seifert fibration.\newline2) The manifold
$M$ is diffeomorphic to $\mathbb{T}^{k}\times P$, where $\mathbb{T}^{k}$ is
the torus of dimension $k>1$ and $P$ is a flat manifold. The flow $F$ is
conjugated to a flow which,  when restricted to the first factor, is a fixed
linear flow, \cite[Theorem 4.2]{Mol}.

In the elliptic case, there are two possibilities:\newline1) All the orbits of $F$ are closed and constitute a Seifert fibration.\newline2) The flow $F$ is obtained by suspension of an isometry of a compact Riemannian manifold with
curvature $1$ (in other words, a manifold covered by the sphere), \cite[Theorem 4.3]{Mol}.

In the elliptic case, if the dimension of $M$ is odd, then the Riemannian flow is a homogeneous $\mathbb{S}^{2q}$-foliation, where  $dim(M)=2q+1$, then the flow has a compact leaf, \cite[Proposition 5.1]{Blu}.

Imposing the causality condition we avoid the existence of closed orbit in the above classification for the Riemannian flow in Corollary \ref{transv constant}, so we can state
the following result.

\begin{theorem}
Let $(M,g)$ be a causal $n$-dimensional Lorentzian
manifold of constant sectional curvature $c$.

\begin{itemize}
\item[1.] If $c>0$ and $n$ is even then there is no compact totally geodesic
null hypersurface.

\item[2.] If $c=0$, then any compact totally geodesic null hypersurface is diffeomorphic to $\mathbb{T}^{k}\times P$, where
$\mathbb{T}^{k}$ is the torus of dimension $k>1$ and where $P$ is a flat manifold.

\item[3.] If $c<0$ and $n\geq6$, then there is no compact totally geodesic
null hypersurface.

\end{itemize}
\end{theorem}

\begin{remark}
 The above theorem holds if instead of the constant curvature condition on
the ambient space, we can find a rigging $\zeta$ for which $K(X,Y)$
is constant for all $X,Y\in\mathcal{S}^{\zeta}$.
\end{remark}
It is known that the fundamental group of a 3-dimensional compact manifold admitting a codimension 2 transversally euclidean flow is solvable, see \cite[Corollary 5.6]{Blu}. Moreover, a $n$-dimensional compact manifold admitting a codimension 2 transversally euclidean flow has non trivial first homology group over integers, i.e, $H_1(M, \mathbb{Z}) \neq 0$, see \cite[Corollary 5.7]{Blu}. So, we have the following.

\begin{proposition}
 Let $(M,g)$ be a 4-dimensional flat spacetime. Then any compact totally geodesic null hypersurface $L$ in $M$ has solvable fundamental group and non trivial first homology group over integers, i.e, $H_1(L, \mathbb{Z}) \neq 0.$
\end{proposition}

\section{Geodesic properties of the rigged metric}
A clear advantage of the rigging technique is the use of the induced Riemannian structure on it. In particular, it would be useful to know when both structures, the rigged metric, and the ambient metric share their geodesics. It can be expected that this happens in some cases despite the fact they have different signatures. This is a little bit more involved than totally geodesic hypersurfaces in Riemannian manifolds, but the results in this directions are enough for our purposes.
In this section we explore this relation on both families of geodesics.

Let $P:span(\xi)\oplus\mathcal{S}^{\zeta}\rightarrow\mathcal{S}^{\zeta}$ be
the canonical projection. Let $\overline{C}\in T_{0}^{2}(M)$ defined by%
\[
\overline{C}(U,V)=C(P(U),P(V))-\omega(U)\tau(P(V))-\omega(V)\tau
(P(U))-\omega(U)\omega(V)\tau(\xi).
\]

Multipliying Equation (2.3) of \cite[p. 1221.]{GB} by $\zeta$ we get
\begin{equation}
C(U,X)=-g(\nabla_{U}\zeta,X)-\frac{1}{2}g(\zeta,\zeta)B(U,X), \label{1}%
\end{equation}
$U\in\mathfrak{X}(L)$ and $X\in\mathcal{S}^{\zeta}$.

If $\zeta$ is a closed rigging for $L$, then for any $U,V\in\mathfrak{X}(L)$,
\cite[Prop. 7, p. 10.]{GutOle21}
\begin{equation}
\widetilde{\nabla}_{U}V=\nabla_{U}^{L}V+\left(  B(U,V)-\overline
{C}(U,V)\right)  \xi. \label{2}
\end{equation}

In particular if $\zeta$ is a closed rigging for a totally geodesic null hypersurface, $\nabla_{U}V=\widetilde{\nabla
}_{U}V$ for any $U,V\in\mathfrak{X}(L)$ if and only if $\overline{C}=0$.
Taking into account the definition of $\overline{C}$ and Equation (\ref{1}), this is equivalent to $\tau=0$ and $\nabla_{X}Y\in \mathcal{S}^{\zeta}$ for any $X,Y\in \mathcal{S}^{\zeta}$.

\begin{proposition}
Let $(M,g)$ be a Lorentzian manifold and $L$ a totally geodesic null hypersurface admitting a closed rigging $\zeta$. A curve $\gamma:I\rightarrow L$ is
both a $g$-geodesic and a $\widetilde{g}$-geodesic if and only if
$\overline{C}(\gamma^{\prime
},\gamma^{\prime})=0$.
\end{proposition}

We can weaken the above conditions to prove the following result, which is useful for our purposes.

\begin{proposition}\label{prop3}
Let $(M,g)$ be a Lorentzian manifold and $L$ a totally geodesic null
hypersurface that admits a rigging $\zeta$ such that the screen distribution is integrable and $C(X,X)=0$ for any $X\in\mathcal{S}^{\zeta}$. Let $\gamma:I\rightarrow S$ be any curve with $S$ a leaf of $\mathcal{S}$. Then the following claims are equivalent.
\begin{itemize}
    \item The curve $\gamma$ is a $\widetilde{g}$-geodesic.
    \item The curve $\gamma$ is a $g$-geodesic.
    \item The curve $\gamma$ is a $i^*g$-geodesic, where $i:S\rightarrow M$ is the canonical inclusion.
\end{itemize}
\end{proposition}

\begin{proof}
It is easy to see that the two formulas in \cite[Proposition 3.13.]{GB} are
true imposing $\mathcal{S}^{\zeta}$ is integrable, not necessarily the rigging to be
closed. Thus, being $L$ totally geodesic which means $B=0$, we have
that $S$ is a totally geodesic submanifold of $L$, that is $\widetilde{\nabla}_{X}Y=\nabla_{X}^{\ast}Y$ for any $X,Y\in\mathcal{S}^{\zeta}$ being
$\nabla^{\ast}$ the connection on the leaves of $\mathcal{S}^{\zeta}$ induced by the ambient connection $\nabla$. This shows the equivalence of points one and three.

Now, we use \cite[Remark 3.8]{GB} to
write%
\[
\nabla_{\gamma^{\prime}}^{L}\gamma^{\prime}-\widetilde{\nabla}_{\gamma
^{\prime}}\gamma^{\prime}=\left(  C(\gamma^{\prime},\gamma^{\prime
})+\widetilde{g}(\widetilde{\nabla}_{\gamma^{\prime}}\xi,\gamma^{\prime
})\right)  \xi.
\]

The first summand is zero by hypothesis. The second one is zero because%
\[
\widetilde{g}(\widetilde{\nabla}_{X}\xi,Y)=\frac{1}{2}\left(  d\omega
(X,Y)+L_{\xi}\widetilde{g}(X,Y)\right)  =0
\]
for any $X,Y\in\mathcal{S}^{\zeta}$. The first summand in this expression is zero because $\mathcal{S}^{\zeta}$ is integrable. The second one is also zero because $L_{\xi}\widetilde{g}=-2B=0$. Therefore
$\gamma$ is a $\nabla^{L}$-geodesic. But $B=0$ is equivalent to $\nabla^L=\nabla$ so $\gamma$ is a $g$-geodesic.

Suppose now that $\gamma$ is a $g$-geodesic, then it is a $\nabla^{L}$-geodesic
because $B=0$, so using the hypothesis we have $0=\nabla_{\gamma^{\prime}}%
^{L}\gamma^{\prime}=\widetilde{\nabla}_{\gamma^{\prime}}\gamma^{\prime}$.

\end{proof}

\begin{proposition}
	\label{prop1}
	Let $(M,g)$ be a chronological 3-dimensional Lorentzian manifold admitting a complete timelike Killing vector field $\zeta$ such that $K(\Pi)\geq 0$ for all timelike plane $\Pi$ containing $\zeta$. Suppose
	 $(M,g)$ is null and spacelike complete. Then any topologically closed embedded totally geodesic null surface $L$ contained in $M$ is complete with respect to the rigged metric $\widetilde{g}$ constructed from the rigging $\zeta$, and the universal covering of $L$ is isometric to $\mathbb{R}^2$ with the euclidean metric.
\end{proposition}
\begin{proof}
Suppose $L$ exists and take $\zeta$ as a rigging. Because $\zeta$ Killing, its rigged vector field is $g$-geodesic, that is $\nabla_{\xi}\xi=0$.
From \cite{RomMig}, and since $\zeta$ is Killing, we have
\begin{equation}
(Hess f)(X,Y) = -g(R(X, \zeta)\zeta, Y) + g(\nabla_X\zeta, \nabla_Y\zeta), \forall X, Y \in TM,
\end{equation}
where $f = \frac{1}{2}g(\zeta, \zeta)$. Take $X = Y = \xi$ and using the fact that $\nabla_\xi\xi = 0$ we get
\begin{equation}
\xi(\xi(f)) = -g(R(\xi, \zeta)\zeta, \xi) + g(\nabla_\xi\zeta, \nabla_\xi\zeta).
\end{equation}

So we get
\begin{equation}
\xi(\xi(f)) = K(\zeta, \xi) + g(\nabla_\xi\zeta, \nabla_\xi\zeta).
\end{equation}

Using again $\zeta$ is Killing we have $g(\nabla_\xi\zeta, \xi) = 0$, which implies $\nabla_\xi\zeta$ is not timelike so
$g(\nabla_\xi\zeta, \nabla_\xi\zeta) \geq 0$. As a consequence $\xi(\xi(f))\geq 0$. Now, take any integral curve $\gamma$ of $\xi$ which we know is a geodesic in $M$. As $(M,g)$ is null complete and $L$ is  topologically closed embedded in $M$, $\gamma$ is defined on $\mathbb{R}.$
But $(f\circ \gamma)^{\prime\prime}\geq 0$, so $f\circ \gamma$ is a convex function defined on $\mathbb{R}$ and is bounded above (since $f $ is negative), it follows that $f\circ \gamma$ is constant and $(f\circ \gamma)^{\prime\prime}= 0$ and then
\begin{equation}
 K(\zeta, \xi) + g(\nabla_\xi\zeta, \nabla_\xi\zeta) = 0.
\end{equation}
Since each term is non negative, we get $ K(\zeta, \xi) = g(\nabla_\xi\zeta, \nabla_\xi\zeta) = 0$. Now since $g(\nabla_\xi\zeta, \xi) = 0$, and $g(\nabla_\xi\zeta, \nabla_\xi\zeta) = 0$, $\nabla_\xi\zeta$ is tangent to $L$ and proportional to $\xi$. But $\tau(X) = g(\nabla_X\zeta, \xi)$, $\forall X \in TL$. Since $\zeta$ is Killing, $g(\nabla_X\zeta, \xi) = -g(\nabla_\xi\zeta, X) = 0$ since $\nabla_\xi\zeta$ is  proportional to $\xi$. So $\tau (X) = 0$, $\forall X \in TL$, $\widetilde{\nabla}_\xi \xi = 0$ and $d\omega(\xi, X)= 0$ for $X $ in the screen distribution (see Lemma 3.10 in \cite{GB}). Since $L$ is two-dimensional, this means that $d\omega = 0$. From \cite[corollary 3.6, point 4]{GB}, $C =0$ since $B = 0, d\omega = 0, L_\zeta g =0.$ From \cite[Equation (20)]{GutOle21}, the connections $\nabla^L, \widetilde{\nabla}$ are equal, and since $B=0$, the connections $\nabla, \widetilde{\nabla}$ are equal. As a consequence, any geodesic of $(L, \widetilde{g})$ is a geodesic of $(M,g)$ and since  $(M,g)$ is null and spacelike complete, then  $(L,\widetilde{g})$ is complete.
From \cite[Theorem 4.4]{GB} with $\mathcal{S^\zeta}$ integrable, $\xi$ $\widetilde{g}$-geodesic and $B=0$, the curvature of $L$ is the null sectional curvature of the null plane $\Pi=span(\xi,X)$ with $X\in\mathcal{S^\zeta}$, but this is zero because $L$ is totally geodesic, see \cite[Equation 2.10]{GB}. So  $(L, \widetilde{g})$ is a complete flat surface. It follows that the universal covering of $L$ is $\mathbb{R}^2$.
\end{proof}

\begin{remark}
Suppose in Proposition \ref{prop1} that $(M,g)$ is $n$ dimensional, then the above proof shows that $\tau (X) = 0, \forall X \in TL$ and $d\omega(\xi, X)= 0$ for $X $ in the screen distribution. Hence if  $(M,g)$ is static with static vector field $\zeta$, then $d\omega = 0$ since in this case $d\omega (X, Y) = 0$ for all $X, Y \in \mathcal{S}^{\zeta}$. Then  the connections $\nabla, \widetilde{\nabla}$ are equal on $\mathfrak{X}(L)$. As a consequence, any geodesic of $(L, \widetilde{g})$ is a geodesic of $(M,g)$ and since  $(M,g)$ is null and spacelike complete, then  $(L, \widetilde{g})$ is complete.
\end{remark}

We can state the following.

\begin{proposition}
	\label{prop2}
	Let $(M,g)$ be a $n$-dimensional  static Lorentzian manifold  with static vector field  $\zeta$ such that $K(\Pi)\geq 0$ for all timelike plane $\Pi$ containing $\zeta$. Suppose
	 $(M,g)$ is null and spacelike complete. Then any topologically closed embedded totally geodesic null hypersurface is complete with respect to the rigged metric $\widetilde{g}$ constructed from the rigging $\zeta$.
\end{proposition}

\section{Existence of periodic geodesics}

The question of the existence of periodic geodesics (both the initial and final conditions of the geodesic coincide, that is, $\gamma'(0)=\gamma'(1)$) in a compact Lorentzian manifold is a long-time problem in Lorentzian geometry, recently solved in the negative in the paper \cite{ABZ}. In the Lorentzian setting there is the additional question of the causal character of the geodesics.

 On the other hand, Guediri has proved that compact flat spacetimes contain a causal periodic geodesic, see \cite{Gued}. To our
knowledge, the constant negative curvature case is still open. Recall that compact Lorentzian manifolds with constant positive curvature do not exist.

In this section we apply the ideas of the above sections to establish several results on the existence of this kind of geodesics, in particular in constant negative curvature. We show that the existence depends not only of the curvature properties, but on the causality of the ambient space.

\begin{theorem}
\label{closed geod} Let $(M,g)$ be a non totally vicious
compact $n$-dimensional spacetime ($n\geq6)$ with constant negative curvature.
Then it contains infinite periodic null geodesics.
\end{theorem}

\begin{proof}
Since $M$ is compact and not totally vicious it contains a null line $\eta$, see \cite[Theorem 12)]{Ming}. Moreover, since it has constant curvature, it is complete (\cite{Car1, Klin}) and satisfies the null convergence condition. This is because constant curvature is equivalente to zero null
sectional curvature, \cite[Prop. 2.3 p. 296.]{Har82}, then $Ricc(u,u)=0$ for
any null vector, see for example \cite[Lemma 5.1, p. 152]{GutOle09}. Using the null completeness and the null convergence condition, the null line $\eta$ is contained in a smooth (topologically) closed achronal totally geodesic null hypersurface $L$ \cite[Theorem~IV.1.]{Ga}. It follows that $L$ is a compact totally geodesic null hypersurface of dimension at least $5$. Since the flow induced by the null generator is transversally hyperbolic, all the orbits are periodic.
\end{proof}

In case there exists on $M$ a closed timelike vector field, we can
relax the constant curvature assumption as well as the dimensional restriction. The restriction $dim M \geq 4$ is needed to leave room to the transverse curvature.

\begin{theorem}
Let $(M,g)$ be a $n\geq 4$-dimensional non totally vicious compact spacetime with negative sectional curvature on spacelike planes. Suppose it is null complete, satisfies the null convergence
condition and admits a closed timelike vector field $\zeta$. Then it has infinite periodic null geodesics.
\end{theorem}

\begin{proof}
Following the proof of Theorem \ref{closed geod}, we can can prove that there exists a compact totally geodesic null hypersurface $L$.
Using $\zeta$ as a closed rigging for $L$, it follows that the rigged vector
field $\xi$ is a Killing (in fact parallel) vector field on $L$ hence the flow defined by $\xi$
is isometric, in particular the foliation defined by $\xi$ is a Killing
foliation. Since the flow have transverse negative curvature, Proposition \ref{curvat}, from \cite[Theorem~7]{Fran} it is a closed foliation which
proves the existence of infinite periodic orbit in $M$.
\end{proof}

In particular, if $M$ is compact and the sectional curvature is constant and negative, it is complete and the null sectional curvature is zero, so it satisfies the null convergence condition as we commented in the proof of Theorem \ref{closed geod}. The presence of a closed timelike vector field allows us to prove Theorem \ref{closed geod} in $dim(M)\geq 4$.

\begin{corollary}
Let $(M,g)$ be a non totally vicious compact
$n$-dimensional spacetime ($n\geq4$) with constant negative curvature. If
there exists a closed timelike vector field then
it has infinite periodic null geodesics.
\end{corollary}

A bound on the ricci curvature also gives the following.

\begin{theorem}
\label{ricc bound} Let $(M,g)$ be a spacetime and $L$ a
compact totally geodesic null hypersuraface. Suppose $\zeta$ is a
closed rigging for $L$ with associate null rigging $N$ and rigged vector field
$\xi$. If there exists a negative constant $c$ such that for all unitary
$X\in\mathcal{S}^{\zeta}$, $Ric(X,X)-2g(R(\xi,X)X,N)\leq c$, then all the orbits of
$\xi$ are closed, so there exist infinite periodic null geodesics.
\end{theorem}

\begin{proof}
Since the rigging $\zeta$ is closed and $L$ is totally geodesic, from
\cite[Corollary 4.9]{GB}, for all $X\in\mathcal{S}^{\zeta},\widetilde{K}%
(X,\xi)=0$. Moreover from Remark \ref{rmq closed}, $K^{T}(X,Y)=\widetilde
{K}(X,Y)$ $\forall\,X,Y\in\mathcal{S}^{\zeta}$. It follows that
$Ric^{T}(X,X)=\widetilde{Ric}(X,X)$. From \cite[Corollary 4.9]{GB},
$\widetilde{Ric}(X,X)=Ric(X,X)-2g(R(\xi,X)X,N)$, hence $Ric^{T}
(X,X)\leq c<0$. The conclusion follows from \cite[Theorem~6]%
{Fran}.
\end{proof}

\begin{theorem}
Let $(M,g)$ be a non totally vicious compact spacetime
with Ricci curvature bounded above by a negative constant for unitary spacelike vectors.
Suppose it admits a parallel null vector field $V$ and a closed vector
field $\zeta$ transverse to $V$. Then the orbits of $V$ are periodic.
\end{theorem}

\begin{proof}
Since $(M,g)$ is non totally vicious, from \cite[Theorem
3.4]{AGR}, $V$ is tangent to a compact totally geodesic null hypersurface $L$. Using $\zeta$ as a rigging for $L$, and following
the proof of Theorem \ref{ricc bound}, the transverse Ricci curvature of the
flow defined by the rigged vector field is given by $Ric^{T}%
(X,X)=Ric(X,X)-2g(R(\xi,X)X,N)$. As $V$ is parallel and proportional to
$\xi$, we have $g(R(\xi,X)X,N)=0$. Hence $Ric^{T}(X,X)=Ric(X,X)$. The conclusion follows
from \cite[Theorem~6]{Fran}.
\end{proof}

\begin{theorem}
Let $(M,g)$ be a $n(\geq 4)$-dimensional non totally vicious compact spacetime with negative sectional curvature on spacelike planes. Suppose it is null complete, NCC holds and there exists a closed timelike vector field. Then there are infinite periodic null geodesics.
\end{theorem}

\begin{proof}
Following the proof of Theorem \ref{closed geod}, we can can prove that there exists a compact totally geodesic null hypersurface $L$. Using the closed timelike vector field $\zeta$ as a rigging for $L$, it follows that the rigged vector field $\xi$ is  Killing, hence the flow defined by $\xi$ is isometric, in particular the foliation defined by $\xi$ is a Killing foliation. Since the flow have transverse negative curvature, from \cite[Theorem 7]{Fran}, it has compact orbits which shows the existence of infinite periodic null geodesics.
\end{proof}

\begin{corollary}
Let $(M,g)$ be a $n(\geq 4)$-dimensional non totally vicious compact spacetime with constant negative
curvature. If there exists on $M$ a closed timelike vector field. Then there exist infinite periodic null geodesics.
\end{corollary}

\begin{theorem}
Let $(M,g)$ be a null and spacelike complete Lorentzian
manifold of dimension $\geq4$ and $L$ a (topological) closed embedded and
totally geodesic null hypersurface with $\zeta$ a closed and orthogonally
Killing rigging with $\nabla_{U}\zeta\in\mathfrak{X}(L)$ for any
$U\in\mathfrak{X}(L)$. Suppose that the sectional curvature of $(M,g)$ is greater than a
positive constant $c$ for any plane $\Pi\subset\mathcal{S}^{\zeta}$. Then $L$
is a quotient of a standard cylinder. Moreover $(M,g)$ has a spatial
periodic geodesic.
\end{theorem}

\begin{proof}
Let $(E_{1}=\xi,E_{2},...,E_{n})$ be a local $\widetilde{g}$-orthonormal basis
in $L$. Now $\zeta$ closed implies $\omega$ closed, and $B=0$ implies that for
any $X\in\mathcal{S}^{\zeta}$
\[
K(\Pi_{i})=\widetilde{K}(\Pi_{i})
\]
being $\Pi_{i}=span\{X,E_{i}\}$, $i=2,...,n$, \cite[Theorem 4.2]{GB}. Thus,%
\[
\widetilde{g}(\widetilde{R}_{E_{i}X}X,E_{i})=g(R_{E_{i}X}X,E_{i}%
),\ \ i=2,..,n.
\]

For $i=1$ we have $\widetilde{g}(\widetilde{R}_{\xi X}X,\xi)=0$ because $\xi$
is $\widetilde{g}$-parallel, \cite[Corollary 3.14]{GB}. This leads us to
$\widetilde{R}ic(X,X)>(n-2)c>0$ for any unitary $X\in\mathcal{S}^{\zeta}$.

Being $L$ closed and totally geodesic, it is complete. In fact, let
$\gamma:(a,b)\rightarrow L$ with $b<\infty$ be any $\widetilde{g}$-geodesic.
The hypothesis imply $\overline{C}(\gamma^{\prime},\gamma^{\prime})=0$, so
$\gamma$ is also a $g$-geodesic. Using $M$ is complete
$\gamma$ has an end in $M$, and $L$ closed embedded implies this
end belong to $L$, and is the end point of $\gamma$ in $L$, so $\gamma$ is
extensible in $(L,\widetilde{g})$.

The universal covering $\pi:N\rightarrow L$ is a direct product $N=\mathbb{R}%
\times P$ with the metric $\pi^{\ast}\widetilde{g}$ where $\partial_{t}$ is
identified with $\xi$ and $P$ with a leaf of $\mathcal{S}^{\zeta}$, see
\cite[Theorem 5.3]{GB}. Then $(P,\pi^{\ast}\widetilde{g})$ is a complete
Riemannian manifold with Ricci tensor greater than a positive constant. Myers'
theorem says that $P$ is compact and being simply connected, it is a sphere.
This prove the first part.

For the second part, observe that $P$ is a Riemannian manifold with
its natural metric ($\pi\circ i)^{\ast}g$ inherited from the ambient $(M,g)$, and being compact it has a closed geodesic
$\gamma$. It proyects to a spatial and closed $i^{\ast}g$-geodesic $\pi
\circ\gamma$ in a leaf of $S$. Using the second fundamental form of this leaf%
\[
\mathbb{I}^{S^{\zeta}}(X,Y)=C(X,Y)\xi+B(X,Y)N
\]
\cite[Equation (2.4)]{GB}, and using $B=0$ and $C(X,X)=0$ for any
$X\in\mathcal{S}^{\zeta}$ we conclude that $\pi\circ\gamma$ is a $g$-geodesic.
\end{proof}

\begin{theorem}
Let $(M,g)$ be a spatially complete Lorentzian manifold admitting a compact
totally geodesic null hypersurface $L$. Suppose it admits a rigging $\zeta$
such that the screen distribution $\mathcal{S}$ is integrable, the curvature
of any plane in $\mathcal{S}$ is greater than $c>0$ and $C(X,X)=0$ for any
$X\in\mathcal{S}$. Then there exists infinite periodic spatial geodesics.
\end{theorem}

\begin{proof}
We can write%
\[
\nabla_{X}^{L}Y-\widetilde{\nabla}_{X}Y=C(X,Y)\xi
\]
\cite[Remark 3.8]{GB}, so given a $\widetilde{g}$-geodesic $\gamma :I\rightarrow L$ with $\gamma
^{\prime}(0)\in\mathcal{S}$, applying Proposition \ref{prop3} we have that it is
a $g$-geodesic, so by the hypothesis, the leaves of $\mathcal{S}$ are complete
with the induced metric $i^{\ast}g$. The condition on the curvature implies,
using Myer's theorem, that the leaves of $\mathcal{S}$ are compact. Then $M$ fibers over $\mathbb{S}^{1}$, see \cite{JohnsonWhitt80}. Take $S$ a fiber, so it is also a leaf of the screen distribution. Because $S$ is compact there exists a periodic $i^*g$-geodesic which is a periodic $g$-geodesic in $M$ by Proposition \ref{prop3}.
\end{proof}

\begin{theorem}
	\label{thm1}
	Let $(M,g)$ be a chronological 3-dimensional Lorentzian manifold admitting a complete timelike killing vector field $\zeta$ such that $K(\Pi)\geq 0$ for all timelike plane $\Pi$ containing $\zeta$. Suppose
	aditionally that the NCC holds on $M$ and $(M,g)$ is null and spacelike complete. Then at least one of the following holds
	\begin{itemize}
		\item $(M,g)$ is causally continuous
		\item  The universal covering of $M$ is diffeomorphic to $\mathbb{R}^3$
	\end{itemize}
\end{theorem}

\begin{proof}
	Suppose $(M,g)$ is not causally continuous. The existence of a complete timelike Killing vector field implies $(M,g)$ is reflecting, see \cite{JS}. So, by definition of causal continuity it cannot be distinguishing and therefore it is not stably causal. From \cite{Ming}, $(M,g)$ contains a null line. Since the NCC holds and $(M,g)$ is null complete, the null splitting theorem, see \cite{Ga}, implies that this null line is contained in an achronal topologically closed embedded totally geodesic null hypersurface (here a surface) $N$ of $M$.
	From Proposition \ref{prop1}, the universal covering of $N$ is $\mathbb{R}^2$. Since $N$ is achronal, from \cite[Theorem 3.4]{IC}, $M$ is diffeomorphic to $\mathbb{R}\times N$ and then the universal covering of $M$ is $\mathbb{R}^3$.
\end{proof}

\begin{corollary}
	\label{cor1}
	Let $(M,g)$ be a chronological stationary 3-dimensional  complete Lorentzian manifold. Suppose $(M,g)$ is flat and contains no periodic geodesic. Then  $(M,g)$ is causally continuous.
\end{corollary}

\begin{proof}
	Suppose $(M,g)$ is not causally continuous. Call $\zeta$ the timelike Killing vector field. Its lift to the universal covering of $M$ is a Killing vector field of the Minkowski space, so it is a translation or an element of $\mathfrak{o}_1(3)$ which implies that $\zeta$ is complete. The hypothesis on the curvature in the above theorem are satisfied, then its proof and Proposition \ref{prop1} show that $M$ is diffeomorphic to $\mathbb{R}\times L$ where $L$ is a complete flat surface. By the Cheeger-Gromoll soul Theorem, see \cite{CG}, if the soul of $L$ is not a point then it contains a closed geodesic which should also be a closed geodesic of $(M,g)$, contradiction since $M$ admits no closed geodesic. So the soul of $L$ is a point and in this case $L$ is diffeomorphic to $\mathbb{R}^2$. It follows that $M$ is diffeomorphic to $\mathbb{R}^3$. But since $(M,g)$ is flat and complete it is isometric to the 3-dimensional Minkowski space which is causally continous. Contradiction.
\end{proof}

\textbf{Acknowledgement.} This work was supported by the  Ministry of Sciences and Innovation of Spain (MICINN) I+D+I, grant number PID2020-118452GB-I00. The first author acknowledge Junta de Andalucía, Proyecto de excelencia PROYEXCEL-00827.




\end{document}